\documentclass[12pt]{article}
\usepackage{amssymb}
\usepackage{amsfonts}
\usepackage{amsmath,latexsym,amssymb,amsfonts,amsbsy, amsthm}
\usepackage[usenames]{color}
\usepackage{xspace,colortbl}
\usepackage{mathrsfs}
\usepackage[colorlinks,linkcolor=blue,citecolor=blue]{hyperref}
\allowdisplaybreaks

\newcommand{\beq}{\begin{equation}}

\newcommand{\eeq}{\end{equation}}
\newcommand{\ben}{\begin{eqnarray}}
\newcommand{\een}{\end{eqnarray}}
\newcommand{\beno}{\begin{eqnarray*}}
\newcommand{\eeno}{\end{eqnarray*}}

\newtheorem{thm}{Theorem}[section]

\newtheorem{lem}[thm]{Lemma}
\newtheorem{prop}[thm]{Proposition}

\renewcommand{\theequation}{\thesection.\arabic{equation}}

\title{\textbf{Harmonic approximation and improvement of flatness in a singular perturbation problem}}
\author{Kelei Wang
\thanks{The author is supported by NSFC No. 11301522. I would like to thank Prof. Fang Hua Lin for simulating discussions which clarify the geometric meanings of some arguments in this paper.}
\\
{\small  Wuhan Institute of Physics and Mathematics,}\\
{\small Chinese Academy of Sciences, Wuhan 430071, China}\\
{\small wangkelei@wipm.ac.cn } \ }
\date{}

\begin{document}
\maketitle
\begin{abstract}
We study the De Giorgi type conjecture, that is, one dimensional symmetry problem for
entire solutions of an two components elliptic system in $\mathbb{R}^n$,
for all $n\geq 2$. We prove that, if a solution $(u,v)$ has a linear
growth at infinity, then it is one dimensional, that is, depending
only on one variable. The main ingredient is an improvement of
flatness estimate, which is achieved by the harmonic approximation
technique adapted in the singularly perturbed situation.
\end{abstract}


\vskip 0.2cm


\renewcommand{\theequation}{\thesection.\arabic{equation}}
\setcounter{equation}{0}


\section{Introduction}
\numberwithin{equation}{section}
 \setcounter{equation}{0}
 In this paper, we continue our study in \cite{W} on the De Giorgi type conjecture, i.e. one dimensional symmetry problem for
solutions of the following two component elliptic system in
$\mathbb{R}^n$:
\begin{equation}\label{equation}
\Delta u=uv^2,\ \ \Delta v=vu^2,\ \ u,v>0~~\text{in}~~\mathbb{R}^n.
\end{equation}
We remove the energy minimizing condition in \cite{W} and prove the
one dimensional symmetry only under the linear growth condition.
More precisely we prove
\begin{thm}\label{main result}
If $(u,v)$ is a solution of the problem \eqref{equation}, and there
exists a constant $C>0$ such that for any $x\in\mathbb{R}^n$,
\begin{equation}\label{linear growth}
u(x)+v(x)\leq C(1+|x|),
\end{equation}
then after a suitable rotation in $\mathbb{R}^n$,
\[u(x)\equiv u(x_n),\ \ \ \ v(x)\equiv v(x_n).\]
\end{thm}

The linear growth condition is sharp, as shown by the examples
constructed in \cite{BTWW}, where $u-v$ is asymptotic to a
homogeneous harmonic polynomial of degree $d\geq 2$. For more discussions on \eqref{equation}, we refer to \cite{blwz, BTWW,C-L 2,NTTV}.

Through suitable rescalings, the problem \eqref{equation} is
closely related to the following singularly perturbed problem (see Theorem \ref{thm preliminary} below)
\begin{equation}\label{equation scaled}
\left\{
\begin{aligned}
 &\Delta u_\kappa=\kappa u_\kappa v_\kappa^2, \\
 &\Delta v_\kappa=\kappa v_\kappa u_\kappa^2,
                          \end{aligned} \right.
\end{equation}
which is used to describe the ``phase separation" phenomena.
When $\kappa\to+\infty$,
the convergence of solutions $(u_\kappa,v_\kappa)$ of
\eqref{equation scaled} and their singular limit were studied by
Caffarelli and Lin \cite{C-L 2}, Noris-Tavares-Terracini-Verzini
\cite{NTTV} and Tavares-Terracini \cite{TT2011} (see also
Dancer-Zhang and the author \cite{DWZ2011}).

 The main ingredient of our proof is an improvement of flatness estimate for the singular perturbation problem
 \eqref{equation scaled}, which is achieved by the blow up (harmonic approximation) technique. This type of arguments, first introduced by De Giorgi in his work on
the regularity of minimal hypersurfaces \cite{D},
 are by now classical in the elliptic regularity theory. It plays an important role in the establishment of many
 $\varepsilon$
-regularity theorems, such as in the theory of stationary varifolds
(cf. Allard \cite{Allard}, see also \cite[Section 6.5]{Lin} for
an account),
 harmonic maps (cf. L. Simon \cite{Simon}) and nonlinear elliptic systems
 (the indirect method, see for example Chen-Wu \cite[Chapter 12]{Chen-Wu}), just to name a few examples.

In singular perturbation problems, Savin's proof of
 the De Giorgi conjecture for Allen-Cahn equation \cite{Savin} also uses an improvement of flatness estimate and the harmonic
 approximation type argument. However, there
 the quantity to be improved is different from the classical \textit{energy} quantity. Indeed, the method developed in \cite{Savin} is mainly on the viscosity (or Krylov-Safonov) side
 and corresponds to the Harnack
  inequality approach to the regularity of minimal hypersurfaces as developed in Caffarelli-Cordoba
  \cite{C-C}.

In this paper we will explore some aspects of harmonic approximation
arguments in the singular perturbation problem \eqref{equation
scaled}, from the \textit{variational} side. Thus in our estimate we still
use an energy type quantity, which is similar to the \textit{excess} used in Allard's regularity theory. In this sense, our method may be viewed as
  a direct generalization of the classical harmonic approximation technique in this singular perturbation
  problem. 
  
However, in order to get a harmonic function in the blow up limit, we use the stationary
   condition arising from the equation, but not the equation \eqref{equation scaled} itself.
Let us first recall the stationary condition. Given a
$\kappa>0$ fixed, any solution of \eqref{equation scaled},
$(u_\kappa,v_\kappa)$ is smooth. Let $Y$ be a smooth vector
field with compact support, then by considering domain
variations in the form
\[u_\kappa^t(x):=u_\kappa(x+tY(x)),\ \ \ \ \ v_\kappa^t(x):=v_\kappa(x+tY(x)),\ \mbox{for}\ |t|\ \mbox{small},\]
we have
\[\frac{d}{dt}\int\left(|\nabla u_\kappa^t(x)|^2+|\nabla v_\kappa^t(x)|^2+\kappa u_\kappa^t(x)^2
v_\kappa^t(x)^2\right)dx\Big|_{t=0}=0.\] Through some integration by parts we
obtain the stationary condition for $(u_\kappa,v_\kappa)$,
\begin{equation}\label{stationary condition 0}
\int\left(|\nabla u_\kappa|^2+|\nabla v_\kappa|^2+\kappa
u_\kappa^2v_\kappa^2\right)\mbox{div}Y
-2DY(\nabla u_\kappa,\nabla
u_\kappa)-2DY(\nabla v_\kappa,\nabla v_\kappa)=0.
\end{equation}
Here $\mbox{div}$ is the divergence operator, and for a function
$u$,
\[DY(\nabla u,\nabla u)=\sum_{\alpha,\beta=1}^n\frac{\partial Y^\alpha}{\partial x_\beta}\frac{\partial u}{\partial x_\alpha}\frac{\partial u}{\partial x_\beta}.\]

For the problem \eqref{equation scaled}, we have better control and
convergence on the energy level, while the equation itself is badly
behaved. This is why we choose to blow up the stationary condition
to get a harmonic function in the limit.
    Here we would like to mention that the stationary condition appears more naturally in some other
singular perturbation problems, such as the Allen-Cahn model (cf.
Hutchinson-Tonegawa \cite{H-T}) and the Ginzburg-Landau model (cf.
Bethuel-Brezis-Orlandi \cite{B-B-O}). In these problems, the
stationary condition is directly linked to the limit problem, i.e.
the stationary condition for varifolds (in the sense of Allard
\cite{Allard}).

In the remaining part of this paper, a solution $(u,v)$ of the
problem \eqref{equation} will be fixed. We use the notation
$a_\kappa=O(b_\kappa)$, if there exists a constant $C$
such that, as $\kappa\to+\infty$,
\[|a_\kappa|\leq Cb_\kappa,\]
and we say $a_\kappa=o(b_\kappa)$ if
\[\lim_{\kappa\to+\infty}\frac{a_\kappa}{b_\kappa}=0.\]
We use $C$ to denote various universal constants, which are
independent of the base point $x\in\mathbb{R}^n$ and the radius $R$.
(In some cases it depends on the solution itself.) It may be
different from line to line. $H^s$ is used to denote the $s-$dimensional
Hausdorff measure.

\section{The improvement of flatness}
\numberwithin{equation}{section}
 \setcounter{equation}{0}

First, to explain why our main Theorem \ref{main result} is related to a singular
perturbation problem, let us recall the following result, which is
essentially \cite[Lemma 5.2]{W}.
\begin{thm}\label{thm preliminary}
For any $\varepsilon>0$, there exists an $R_0$ such that if $R\geq
R_0$ and $x_0\in\{u=v\}$, by defining
\begin{equation}\label{scaled solutions}
u_\kappa(x):=\frac{1}{R}u(x_0+Rx),\ \ \ \ v_\kappa(x):=\frac{1}{R}v(x_0+Rx),
\end{equation}
there exists a constant $c_0$ independent of
$x_0\in\{u=v\}$ and $R$, and a vector $e$ satisfying $|e|\geq c_0$,
such that
\[\int_{B_1(0)}|\nabla u_\kappa-\nabla v_\kappa-e|^2\leq\varepsilon^2.\]
\end{thm}
Note that $(u_\kappa,v_\kappa)$ satisfies \eqref{equation scaled} with
$\kappa=R^4$. For a proof of this theorem see \cite[Lemma 5.2]{W}. The only new point is that, if
\[\sup_{B_2(0)}|u_\kappa-v_\kappa-e\cdot x|^2\leq\varepsilon^2,\]
then $u_\kappa-v_\kappa$ is also close to $e\cdot x$ in $H^1(B_1(0))$ topology.
 This can be proved by a contradiction argument, using the $H^1$ strong convergence for solutions
  of \eqref{equation scaled} (cf. \cite[Theorem 1.2]{NTTV}). Note that we can replace the global uniform H\"{o}lder estimate used in \cite{NTTV} by the interior uniform H\"{o}lder estimate \cite[Theorem 2.6]{W}.

Throughout this paper, $(u_\kappa,v_\kappa)$ always denotes a solution
defined as in \eqref{scaled solutions}. The following improvement of
decay estimate will be the main ingredient in our proof of Theorem
\ref{main result}.
\begin{thm}\label{thm decay estimate}
There exist four universal constants $\theta\in(0,1/2)$, $\varepsilon_0$ small and $K_0, C(n)$ large such that, if $(u_\kappa,v_\kappa)$ is a solution of \eqref{equation scaled} in $B_1(0)$, satisfying
\begin{equation}\label{small condition}
\int_{B_1(0)}|\nabla u_\kappa-\nabla v_\kappa-e|^2=\varepsilon^2\leq \varepsilon_0^2,
\end{equation}
where $e$ is a vector satisfying $|e|\geq c_0/2$, and $\kappa^{1/4}\varepsilon^2\geq K_0$, then there exists another vector $\tilde{e}$, with
\[|\tilde{e}-e|\leq C(n)\varepsilon,\]
such that
\[\theta^{-n}\int_{B_\theta(0)}|\nabla u_\kappa-\nabla v_\kappa-\tilde{e}|^2\leq \frac{1}{2}\varepsilon^2.\]
\end{thm}
The proof will be given later. Note that this theorem is not a local
result. It depends on the global Lipschitz estimate established in
\cite{W}, which is stated for solutions of \eqref{equation} defined
on the entire space $\mathbb{R}^n$. 

This decay estimate can be used to prove
\begin{thm}
There exists a constant $C>0$ such that, for any $x\in\{u=v\}$ and
$R>1$, there exists a vector $e_{x,R}$, with
\[|e_{x,R}|\geq c_0/2,\]
such that
\[\int_{B_R(x)}|\nabla u-\nabla v-e_{x,R}|^2\leq CR^{n-1}.\]
\end{thm}
\begin{proof}
Fix an $R>1$ and $x_0\in\{u=v\}$, which we assume to be the origin $0$. For each $i>0$, denote
\[R_i:=R\theta^{-i}.\]
Let $E_i$ and the vector $e_i$ be defined by
\[E_i:=\min_{e\in\mathbb{R}^n}R_i^{1-n}\int_{B_{R_i}(0)}|\nabla u-\nabla v-e|^2=R_i^{1-n}\int_{B_{R_i}(0)}|\nabla u-\nabla v-e_i|^2.\]

Note that for any fixed $e$,
\begin{eqnarray*}
R_i^{1-n}\int_{B_{R_i}(0)}|\nabla u-\nabla v-e|^2&\leq& R_i^{1-n}\int_{B_{\theta^{-1}R_i}(0)}|\nabla u-\nabla v-e|^2\\
&\leq& \theta^{1-n}\left(\theta^{-1}R_i\right)^{1-n}\int_{B_{\theta^{-1}R_i}(0)}|\nabla u-\nabla v-e|^2.
\end{eqnarray*}
Hence we always have
\begin{equation}\label{1.1}
E_i\leq\theta^{1-n}E_{i+1}.
\end{equation}
Furthermore, since (see \cite[Theorem 5.1]{W})
$$\sup_{\mathbb{R}^n}\left(|\nabla u|+|\nabla v|\right)<+\infty,$$
there exists a constant $C$, which is independent of $i$, such that
\begin{equation}\label{1.2}
E_i\leq C\theta^{-i}.
\end{equation}

By Theorem \ref{thm preliminary}, for any sequence $i\to+\infty$, there exists a subsequence (still denoted by $i$) such that
\[u_i(x):=R_i^{-1}u(R_ix)\to \left(e\cdot x\right)^+,\ \ \ \ v_i(x):=R_i^{-1}v(R_ix)\to \left(e\cdot x\right)^-.\]
Here $e$ is a vector in $\mathbb{R}^n$ satisfying $|e|\geq c_0$, and
the convergence is in $C_{loc}(\mathbb{R}^n)$ and also in
$H^1_{loc}(\mathbb{R}^n)$. Note that $(u_i,v_i)$ satisfies \eqref{equation scaled} with $\kappa_i=R_i^4$.

Indeed, by Theorem \ref{thm preliminary}, if $R_i\geq R_0$, where $R_0$ is
a constant depending only on $\varepsilon_0$, there exists a vector
$\bar{e}_i$ with $|\bar{e}_i|\geq c_0$ such that
\[\int_{B_{R_i}(0)}|\nabla u-\nabla v-\bar{e}_i|^2\leq \varepsilon_0^2 R_i^n.\]
By definition, if we replace $\bar{e}_i$ by $e_i$, we can get the same estimate. Thus by Theorem \ref{thm decay estimate}, if we also have
$E_{i+1}\geq K_0$, or equivalently,
\[\int_{B_1(0)}|\nabla u_{i+1}-\nabla v_{i+1}-e_{i+1}|^2\geq K_0R_{i+1}^{-1}=K_0\kappa_{i+1}^{-\frac{1}{4}},\]
then there exists another vector $\tilde{e}_{i+1}$ so that
\[\theta^{-n}\int_{B_\theta(0)}|\nabla u_{i+1}-\nabla v_{i+1}-\tilde{e}_{i+1}|^2\leq\frac{1}{2}\int_{B_1(0)}|\nabla u_{i+1}-\nabla v_{i+1}-e_{i+1}|^2.\]
This can be rewritten as
\begin{eqnarray}\label{1.3}
E_i&\leq& R_i^{1-n}\int_{B_{R_i}(0)}|\nabla u-\nabla v-\tilde{e}_{i+1}|^2\nonumber\\
&=&\theta^{-n}R_i\int_{B_\theta(0)}|\nabla u_{i+1}-\nabla v_{i+1}-\tilde{e}_{i+1}|^2\nonumber\\
&\leq&\frac{R_i}{2}\int_{B_1(0)}|\nabla u_{i+1}-\nabla v_{i+1}-e_{i+1}|^2\\ \nonumber
&=&\frac{1}{2}R_iR_{i+1}^{-n}\int_{B_{R_{i+1}}(0)}|\nabla u-\nabla v-e_{i+1}|^2\\ \nonumber
&=&\frac{\theta}{2}E_{i+1}.
\end{eqnarray}

Now we claim that for all $i\geq \min\{\frac{\log R_0-\log
R}{|\log\theta_0|},1\}$,
\begin{equation}\label{1.4}
E_i\leq \theta^{1-n}K_0.
\end{equation}
Assume by the contrary, there exists an $i_0\geq \min\{\frac{\log
R_0-\log R}{|\log\theta_0|},1\}$ such that
$E_{i_0}>\theta^{1-n}K_0$. First by \eqref{1.1},
\[E_{i_0+1}\geq\theta^{n-1}E_{i_0}>K_0.\]
Thus the assumptions of Theorem \ref{thm decay estimate} are satisfied and we have \eqref{1.3}, which says
\[E_{i_0+1}\geq\frac{2}{\theta}E_{i_0}\geq E_{i_0}>\theta^{1-n}K_0.\]
This can be iterated, and we get, for any $j\geq 0$,
\[E_{i_0+j+1}\geq\frac{2}{\theta}E_{i_0+j}\geq E_{i_0}\left(\frac{2}{\theta}\right)^{j+1}.\]
However, since $2/\theta>1/\theta$, this contradicts \eqref{1.2} if
$j$ is large enough.
 Note that the constant
$\theta^{1-n}K_0$ in \eqref{1.4} is independent of the base point
$x_0\in\{u=v\}$ and the radius $R$.
 Thus we get \eqref{1.4} for any $R\geq R_0$ and
$x\in\{u=v\}$. Then by choosing a larger constant, this can be extended to cover $[1,R_0]$ if we note the
global Lipschitz bound of $u$ and $v$.

We have shown the existence of $e_{x,R}$ for any $x\in\{u=v\}$ and
$R>1$. The lower bound for $|e_{x,R}|$ can be proved as in the proof
of Theorem \ref{thm preliminary} by using the
Alt-Caffarelli-Friedman inequality (see \cite[Theorem 4.3]{W}).
\end{proof}

With this theorem in hand, we can use $e_{x,R}\cdot (y-x)$ to
replace the harmonic replacement $\varphi_{R,x}$ in \cite[Section
7]{W}. The following arguments to prove Theorem \ref{main result}
are exactly the same one in \cite[Section 8 and 9]{W}.

The remaining part of this paper will be devoted to the proof of
Theorem \ref{thm decay estimate}.

\section{Some a priori estimates}
\numberwithin{equation}{section}
 \setcounter{equation}{0}

In this section, we present some a priori estimates for the solution
$(u,v)$. These estimates show that various quantities, when
integrated on $B_R(x)$, have a growth bound as $R^{n-1}$. This is
exactly what we expect for one dimensional solutions. Several
estimates from \cite{W} will be needed in this section.
\begin{lem}\label{lem measure estimate for level sets}
There exist two positive constants $C$ and $M$, such that for any $R>CM$ and $t\geq M$,
\[H^{n-1}(B_R\cap\{u=t\})\leq CR^{n-1}.\]
\end{lem}
\begin{proof}
First, by \cite[Lemma 5.2 and Lemma 5.4]{W}, there exists a constant
$c(M)>0$ such that,
\begin{equation}\label{2.0}
|\nabla u|\geq c(M)\ \ \ \mbox{on}\ \{u=t\}.
\end{equation}
Since $u$ is smooth, by the implicit function theorem $\{u=t\}$ is a
smooth hypersurface.

Now
\begin{equation}\label{2.1}
H^{n-1}(B_R\cap\{u=t\})\leq c(M)^{-1}\int_{B_R\cap\{u=t\}}|\nabla u|.
\end{equation}
Note that on $\{u=t\}$
\[|\nabla u|=-\frac{\partial u}{\partial\nu},\]
where $\nu$ is the unit normal vector of $\{u=t\}$ pointing to $\{u<t\}$. Then by the divergence theorem
\begin{eqnarray*}
-\int_{B_R\cap\{u=t\}}\frac{\partial u}{\partial\nu}&=&\int_{\partial B_R\cap\{u>t\}}\frac{\partial u}{\partial r}-\int_{B_R\cap\{u>t\}}\Delta u\\
&\leq&\int_{\partial B_R\cap\{u>t\}}|\nabla u|+\int_{B_R}\Delta u\\
&\leq&2\int_{\partial B_R}|\nabla u|\leq CR^{n-1}.
\end{eqnarray*}
Here we have used the global Lipschitz continuity of $u$, cf.
\cite[Theorem 5.1]{W}.
\end{proof}
The same results also hold for $v$ and $u-v$, which we do not repeat
here. Next we give a measure estimate for the transition part
$\{u\leq T, v\leq T\}$.
\begin{lem}\label{lem 2.2}
For any $T>1$, there exists a constant $C(T)>0$, such that for any $R>1$ and $x\in\mathbb{R}^n$,
\[H^n(B_R(x)\cap\{u\leq T, v\leq T\})\leq C(T)R^{n-1}.\]
\end{lem}
\begin{proof}
First we have the\\
{\bf Claim.} For each $T>1$, there exists a $c(T)>0$ such that, if $x_0\in\{u\leq T, v\leq T\}$, then
\begin{equation}\label{2.2}
u(x_0)\geq c(T),\ \ \ \ v(x_0)\geq c(T).
\end{equation}
By assuming this claim, we get
\[H^n(B_R(x)\cap\{u\leq T, v\leq T\})\leq c(T)^{-4}\int_{B_R(x)\cap\{u\leq T, v\leq T\}}u^2v^2\leq C(T)R^{n-1},\]
where in the last inequality we have used \cite[Lemma 6.4]{W}.

To prove the claim, first we note that, there exists a constant $C_1(T)$ such that
\begin{equation}\label{2.3}
\mbox{dist}(x_0,\{u=v\})\leq C_1(T).
\end{equation}
Indeed, if $\mbox{dist}(x_0,\{u=v\})\geq L$ ($L$ large to be chosen), take $y_0\in\{u=v\}$ to realize this distance and define
\[\tilde{u}(x)=\frac{1}{L}u(y_0+Lx),\ \ \ \ \tilde{v}(x)=\frac{1}{L}v(y_0+Lx).\]
Then by \cite[Lemma 5.2]{W}, there exits a vector $e$ and a universal constant $C$, with
\[\frac{1}{C}\leq |e|\leq C,\]
such that
\[|\tilde{u}(x)-\left(e\cdot x\right)^+|+|\tilde{v}(x)-\left(e\cdot x\right)^-|\leq h(L),\]
where $h(L)$ is small if $L$ large enough.

Without loss of generality we can assume
$B_1(L^{-1}(x_0-y_0))\subset\{\tilde{u}>\tilde{v}\}$. By a geometric
consideration, we have
\[L^{-1}(x_0-y_0)\cdot e\geq \frac{1}{2C}.\]
Consequently,
\[\tilde{u}(L^{-1}(x_0-y_0))\geq L^{-1}(x_0-y_0)\cdot e-h(L)\geq\frac{1}{4C}.\]
Thus $u(x_0)>T$ if $L$ large, which is a contradiction.

After establishing \eqref{2.3}, we can use the standard Harnack inequality and \cite[Lemma 4.7]{W} to deduce the claimed \eqref{2.2}.
\end{proof}

\begin{lem}\label{lem 2.3}
There exists a constant $C>0$, such that for any $R>1$ and $x\in\mathbb{R}^n$,
\[\int_{B_R(x)}|\nabla u||\nabla v|\leq CR^{n-1}.\]
\end{lem}
\begin{proof}
Fix a $T>0$, which will be determined below. (It is independent of $x$ and $R$.) We divide the estimate into three parts, $\{u\leq T, v\leq T\}$, $\{u>T\}$ and $\{v>T\}$. Note that if $T$ is large enough, by \cite[Lemma 6.1]{W}, these three parts are disjoint.

First in $B_R(x)\cap\{u\leq T, v\leq T\}$, by the global Lipschitz continuity of
$u$ and $v$ \cite[Theorem 5.1]{W} and the previous lemma, we have
\begin{equation}\label{2.4}
\int_{B_R(x)\cap\{u\leq T, v\leq T\}}|\nabla u||\nabla v|\leq CH^n(B_R(x)\cap\{u\leq T, v\leq T\})\leq CR^{n-1}.
\end{equation}

If $T$ large, in $\{u>T\}$, $|\nabla u|\geq c(T)>0$ for a constant $c(T)$ depending only on $T$ (cf. the proof of Lemma \ref{lem measure estimate for level sets}). Furthermore, by the proof of \cite[Lemma 6.3]{W}, there exists a constant $C$ such that
\[|\nabla v|\leq Ce^{-\frac{u}{C}}\ \ \ \mbox{in}\ \{u>T\}.\]
Then by the co-area formula and Lemma \ref{lem measure estimate for level sets},
\begin{eqnarray*}
\int_{B_R(x)\cap\{u>T\}}|\nabla u||\nabla v|&=&\int_{T}^{+\infty}\left(\int_{B_R\cap\{u=t\}}|\nabla v|\right)dt\\
&\leq&C\int_{T}^{+\infty}e^{-ct} H^{n-1}(B_R\cap\{u=t\})dt\\
&\leq& CR^{n-1}.
\end{eqnarray*}
The same estimate holds for $\{v>T\}$. Putting these together we can finish the proof.
\end{proof}

\begin{lem}\label{lem 2.4}
There exists a constant $C>0$, such that for any $R>1$ and $x\in\mathbb{R}^n$,
\[\int_{B_R(x)}uv^3+vu^3\leq CR^{n-1}.\]
\end{lem}
\begin{proof}
We still choose a $T>0$, which will be determined below,  and divide the estimate into three parts, $\{u\leq T, v\leq T\}$, $\{u>T\}$ and $\{v>T\}$.

By the proof of \cite[Lemma 6.1]{W}, we still have
\[uv^3+vu^3\leq C\ \ \ \mbox{in}\ \mathbb{R}^n.\]
Then in $B_R(x)\cap\{u\leq T, v\leq T\}$,
\[\int_{B_R(x)\cap\{u\leq T, v\leq T\}}uv^3+vu^3\leq CH^n(B_R(x)\cap\{u\leq T, v\leq T\})\leq CR^{n-1}.\]

Next, by the proof of \cite[Lemma 6.1]{W}, there exists a constant $C$ such that
\[uv^3+vu^3\leq Ce^{-\frac{u}{C}}\ \ \ \mbox{in}\ \{u>T\}.\]
Then by the co-area formula and the lower bound of $|\nabla u|$ in $\{u>T\}$ (i.e. \eqref{2.0}),
\begin{eqnarray*}
\int_{B_R(x)\cap\{u>T\}}uv^3+vu^3&=&\int_{M}^{+\infty}\left(\int_{B_R\cap\{u=t\}}\frac{uv^3+vu^3}{|\nabla u|}\right)dt\\
&\leq&C\int_{M}^{+\infty}e^{-ct} H^{n-1}(B_R\cap\{u=t\})dt\\
&\leq& CR^{n-1}.
\end{eqnarray*}
The same estimate holds for $\{v>T\}$. Putting these together we can finish the proof.
\end{proof}

\section{Blow up the stationary condition}
\numberwithin{equation}{section}
 \setcounter{equation}{0}

In this section and the next one, we prove Theorem \ref{thm decay estimate}. We argue by contradiction, so assume that as $\kappa\to+\infty$, there exists a sequence of solutions $(u_\kappa,v_\kappa)$ satisfying the conditions but not the conclusions in that theorem, that is,
\begin{equation}\label{assumption 1}
\int_{B_1(0)}|\nabla u_\kappa-\nabla v_\kappa-e|^2=\varepsilon_\kappa^2\to0,
\end{equation}
where $e$ is a vector satisfying $|e|\geq c_0$ (Without loss of generality, we can assume that $e=(0,\cdots,0,|e|)=|e|e_n$), but for any vector $\tilde{e}$ satisfying (here the constant $C(n)$ will be determined later)
\begin{equation}\label{assumption 4}
|\tilde{e}-e|\leq C(n)\varepsilon_\kappa,
\end{equation}
we must have ($\theta$ will be determined later)
\begin{equation}\label{assumption 2}
\theta^{-n}\int_{B_\theta(0)}|\nabla u_\kappa-\nabla v_\kappa-\tilde{e}|^2\geq\frac{1}{2}\varepsilon_\kappa^2.
\end{equation}
Moreover, we also assume that,
\begin{equation}\label{assumption 3}
\lim_{\kappa\to+\infty}\kappa^{1/4}\varepsilon_{\kappa}^2=+\infty.
\end{equation}
We will derive a contradiction from these assumptions.

The first step, which will be done in this section, is to show that
the blow up sequence
\[\frac{u_\kappa-v_\kappa-e\cdot x}{\varepsilon_\kappa},\]
converges to a harmonic function in some weak sense.

Recall that $(u_\kappa,v_\kappa)$ satisfies the stationary condition
\begin{equation}\label{stationary condition}
\int_{B_1(0)}\left(|\nabla u_\kappa|^2+|\nabla v_\kappa|^2+\kappa u_\kappa^2v_\kappa^2\right)\mbox{div}Y
-2DY(\nabla u_\kappa,\nabla u_\kappa)-2DY(\nabla v_\kappa,\nabla v_\kappa)=0.
\end{equation}

Since there exists an $R>0$ such that $\kappa=R^4$, and
\[u_\kappa(x)=R^{-1}u(Rx),\ \ \ \ v_\kappa(x)=R^{-1}v(Rx),\]
by \cite[Lemma 6.4]{W},
\[\int_{B_1(0)}\kappa u_\kappa^2v_\kappa^2=R^{-n}\int_{B_R(0)}u^2v^2\leq CR^{-1}=C\kappa^{-\frac{1}{4}}.\]
Similarly, by Lemma \ref{lem 2.3},
\[\int_{B_1(0)}|\nabla u_\kappa||\nabla v_\kappa|=R^{-n}\int_{B_R(0)}|\nabla u||\nabla v|\leq C\kappa^{-\frac{1}{4}}.\]
Then by a direct expansion, we get
\begin{equation}\label{stationary condition 2}
\int_{B_1(0)}|\nabla(u_\kappa-v_\kappa)|^2\mbox{div}Y-2DY(\nabla (u_\kappa-v_\kappa),\nabla(u_\kappa-v_\kappa))=O(\kappa^{-\frac{1}{4}}).
\end{equation}

Now let
\[w_\kappa:=\frac{u_\kappa-v_\kappa-e\cdot x}{\varepsilon_\kappa}-\lambda_\kappa,\]
where $\lambda_\kappa$ is chosen so that
\begin{equation}\label{mean value normalization}
\int_{B_1(0)}w_\kappa=0.
\end{equation}
Then
\[\int_{B_1(0)}|\nabla w_\kappa|^2=1,\]
and by noting \eqref{mean value normalization} we can apply the Poincare inequality to get
\[
\int_{B_1(0)}w_\kappa^2\leq C(n).
\]
Hence after passing to a subsequence of $\kappa$, we can assume that $w_\kappa$ converges to $w$,
weakly in $H^1(B_1(0))$ and strongly in $L^2(B_1(0))$.

Substituting $w_\kappa$ into \eqref{stationary condition 2}, we obtain
\begin{eqnarray}\label{3.1}
\nonumber O(\kappa^{-\frac{1}{4}})&=&\varepsilon_\kappa^2\int_{B_1(0)}\left[|\nabla w_\kappa|^2\mbox{div}Y-2DY(\nabla w_\kappa,\nabla w_\kappa)\right]\\
&&+2\varepsilon_\kappa\int_{B_1(0)}\left[\nabla w_\kappa\cdot e\mbox{div}Y-DY(\nabla w_\kappa,e)-DY(e,\nabla w_\kappa)\right]\\
\nonumber &&+\int_{B_1(0)}\left[|e|^2\mbox{div}Y+2DY(e,e)\right].
\end{eqnarray}
The last integral equals $0$. Integrating by
parts, we also have
\[\int_{B_1(0)}\left[\nabla w_\kappa\cdot e\mbox{div}Y-DY(\nabla w_\kappa,e)-DY(e,\nabla w_\kappa)\right]
=|e|\int_{B_1(0)}Y^n\Delta w_\kappa.\]
Substituting this into \eqref{3.1}, we obtain
\begin{equation}\label{2}
2|e|\int_{B_1(0)}Y^n\Delta w_\kappa
=-
\varepsilon_\kappa\int_{B_1(0)}\left[|\nabla
w_\kappa|^2\mbox{div}Y-2DY(\nabla w_\kappa,\nabla
w_\kappa)\right]+O(\kappa^{-\frac{1}{4}}\varepsilon_\kappa^{-1}).
\end{equation}
The right hand sides goes to $0$ as $\kappa\to0$, thanks to our
assumptions that $\varepsilon_\kappa\to0$ and
$\kappa^{1/4}\varepsilon_\kappa\to+\infty$. After passing to the
limit in the above equality, we see
\[\int_{B_1(0)}Y^n\Delta w=0.\]
Since $Y^n$ can be any function in $C_0^\infty(B_1(0))$, by standard elliptic theory we get
\begin{prop}
$w$ is a harmonic function.
\end{prop}

\section{Strong convergence of the blow up sequence}
\numberwithin{equation}{section}
 \setcounter{equation}{0}

In this section, we prove the strong convergence of $w_\kappa$
in $H^1_{loc}(B_1(0))$. With some standard estimates on harmonic
functions, this will give the decay estimate Theorem \ref{thm decay
estimate}.

In order to prove the strong convergence of $w_\kappa$ in $H^1_{loc}(B_1(0))$, we define the defect measure $\mu$ by
\[|\nabla w_\kappa|^2dx\rightharpoonup|\nabla w|^2dx+\mu\ \ \ \ \mbox{weakly as measures.}\]
By the weak convergence of $w_\kappa$ in $H^1(B_1(0))$, $\mu$
is a positive Radon measure. Furthermore, $w_\kappa$ converges
strongly in $H^1_{loc}(B_1(0))$ if and only if $\mu=0$ in $B_1(0)$.

First we note the fact that
\begin{lem}\label{lem support of defect measure}
The support of $\mu$ lies in the hyperplane $\{x_n=0\}$.
\end{lem}
\begin{proof}
By \cite[Theorem 2.7]{W} and \eqref{assumption 1}, as $\kappa\to+\infty$, $(u_\kappa,v_\kappa)$ converges to
$|e|(x_n^+,x_n^-)$ uniformly in $B_1(0)$. For any $h>0$, if $\kappa$ large,
\[u_\kappa\geq h,\ \ \ v_\kappa\leq h\ \ \ \ \mbox{in}\ \{x_n>2|e|^{-1}h\}.\]
Then
\[\Delta v_\kappa\geq \kappa h^2v_\kappa\ \ \  \mbox{in}\ \{x_n>2|e|^{-1}h\},\]
and by \cite[Lemma 4.4]{C-T-V 3},
\[v_\kappa\leq C(n)he^{-\frac{\kappa^{1/2}h^2}{C(n)}}\ \ \  \mbox{in}\ \{x_n>3|e|^{-1}h\}.\]

Note that $u_\kappa$ and $v_\kappa$ are uniformly bounded in
$B_1(0)$ because they are nonnegative, subharmonic. Then by
definition
\[
\Delta w_\kappa=\varepsilon_\kappa^{-1}\left[\kappa v_\kappa^2u_\kappa-\kappa u_\kappa^2 v_\kappa\right]
=O(\varepsilon_\kappa^{-1}e^{-\frac{\kappa^{1/2}h^2}{C(n)}})\to 0,\]
uniformly in $\{x_n\geq 4|e|^{-1}h\}$, thanks to our assumption that $\varepsilon_\kappa\gg\kappa^{-1/8}$. Applying standard interior $W^{2,2}$ estimates, together with the assumption \eqref{assumption 1},  and noting that $w_\kappa$ is uniformly bounded in $H^1_{loc}(B_1(0))$, we see
\[\int_{B_{4/5}(0)\cap\{x_n\geq 5|e|^{-1}h\}}|D^2w_\kappa|^2\]
is uniformly bounded. By Rellich compactness theorem, $\nabla
w_\kappa$ converges to $\nabla w$ strongly in
$L^2_{loc}(B_{4/5}\cap\{x_n\geq 5h\})$. In other words, the support
of $\mu$ lies in $\{x_n\leq 5|e|^{-1}h\}$. We can get the other side
estimate and also let $h\to0$ to finish the proof.
\end{proof}

For any $\eta\in C_0^\infty(B_1(0))$, by Lemma \ref{lem 2.4} and \cite[Lemma 6.4]{W},
\begin{eqnarray}\label{caccioppoli inequality 1}
&&\int_{B_1(0)}\Delta(u_\kappa-v_\kappa)(u_\kappa-v_\kappa)\eta^2\nonumber\\
&=&\int_{B_1(0)}\left(2\kappa u_\kappa^2v_\kappa^2-\kappa u_\kappa v_\kappa^3-\kappa v_\kappa u_\kappa^3\right)\eta^2\\
&=&R^{-n}\int_{B_R(0)}\left[2 u(y)^2v(y)^2-u(y) v(y)^3-v(y) u(y)^3\right]\eta(R^{-1}y)^2dy\nonumber\\
&=&O(R^{-1})=O(\kappa^{-1/4}).\nonumber
\end{eqnarray}
Substituting
\[\Delta(u_\kappa-v_\kappa)=\varepsilon_\kappa\Delta w_\kappa,
\ \ \
u_\kappa-v_\kappa=|e|x_n+\varepsilon_\kappa w_\kappa+\varepsilon_\kappa\lambda_\kappa,\]
into this, we get
\begin{equation}\label{1}
\int_{B_1(0)}\varepsilon_\kappa^2\Delta w_\kappa
w_{\kappa}\eta^2
+\varepsilon_\kappa|e|\Delta w_\kappa\eta^2
x_n+\varepsilon_\kappa^2\lambda_\kappa\Delta w_\kappa\eta^2=O(\kappa^{-1/4}).
\end{equation}
On the other hand, by taking $Y=(0,\cdots,0, \eta^2x_n)$ in
\eqref{2}, we have \footnote{This inequality could be understood as
a Caccioppoli type inequality, which is similar to the one in
Allard's regularity theory for stationary varifolds, see \cite[Lemma
8.11]{Allard} and \cite[Lemma 6.5.5]{Lin}. The choice of $Y$ has a
more direct geometric meaning (as a vector field in normal
directions) in that setting. However, we note that
\eqref{caccioppoli inequality 1} can also be understood as a
Caccioppoli type inequality, as in (for example) De
Giorgi-Nash-Moser theory for linear elliptic equations in divergence
form.}
\begin{eqnarray}\label{caccioppoli inequality 2}
&&2|e|\int_{B_1(0)}\Delta w_\kappa\eta^2 x_n\\
&=&-\varepsilon_\kappa\int_{B_1(0)}|\nabla w_\kappa|^2\left(\eta^2+2\eta\frac{\partial\eta}{\partial x_n}x_n\right)-2\sum_{i=1}^n\frac{\partial w_\kappa}{\partial x_n}\frac{\partial w_\kappa}{\partial x_i}\frac{\partial}{\partial x_i}\left(\eta^2x_n\right)
+O(\kappa^{-\frac{1}{4}}\varepsilon_\kappa^{-1}).\nonumber
\end{eqnarray}
Substituting this into \eqref{1} we see
\begin{eqnarray}\label{9}
&&\int_{B_1(0)}\Delta w_\kappa
w_{\kappa}\eta^2-\frac{1}{2}|\nabla w_\kappa|^2
\left(\eta^2+2\eta\frac{\partial\eta}{\partial
x_n}x_n\right)\\
&&\nonumber\ \ \ \ \ +\sum_{i=1}^n
\frac{\partial w_\kappa}{\partial
x_n}\frac{\partial w_\kappa}{\partial
x_i}\frac{\partial}{\partial
x_i}\left(\eta^2x_n\right)+\lambda_\kappa\Delta w_\kappa\eta^2=O(\kappa^{-\frac{1}{4}}\varepsilon_\kappa^{-2}).
\end{eqnarray}

Concerning the last term in this integral we claim that
\begin{lem}
For any $\eta\in C_0^\infty(B_1(0))$,
\[\lim_{\kappa\to+\infty}\int_{B_1(0)}\lambda_\kappa\Delta w_\kappa\eta^2=0.\]
\end{lem}
\begin{proof}
By the definition of $\lambda_\kappa$ (see \eqref{mean value
normalization}),
\begin{eqnarray*}
\lambda_\kappa
&=&\frac{1}{\varepsilon_\kappa H^n(B_1(0))}\int_{B_1(0)}\left(u_\kappa-v_\kappa-e\cdot x\right)\\
&=&\frac{1}{\varepsilon_\kappa
H^n(B_1(0))}\int_{B_1(0)}\left(u_\kappa-v_\kappa\right).
\end{eqnarray*}
By Theorem \ref{thm preliminary}, $u_\kappa-v_\kappa$ converges
uniformly on $B_1(0)$ to the harmonic function $e\cdot x$, thanks to
the assumption that we always have (see the definition of $u_\kappa$
and $v_\kappa$, \eqref{scaled solutions})
\[u_\kappa(0)-v_\kappa(0)=0.\]
This then implies that
\[\lim_{\kappa\to+\infty}\int_{B_1(0)}\left(u_\kappa-v_\kappa\right)=0.\]
Hence we have
\begin{equation}\label{7}
\lambda_\kappa=o(\varepsilon_\kappa^{-1}).
\end{equation}

On the other hand, by choosing the vector field
$Y=(0,\cdots,0,\eta^2)$ in \eqref{2}, we have
\begin{equation}\label{8}
\int_{B_1(0)}\eta^2\Delta w_\kappa=O(\varepsilon_\kappa),
\end{equation}
where we have used the uniform bound on
$\|w_\kappa\|_{H^1(B_1(0))}$ and our assumption
\eqref{assumption 3}.

Combining \eqref{7} and \eqref{8} we get the required convergence.
\end{proof}

Note that \eqref{assumption 3} says
$\kappa^{-\frac{1}{4}}\varepsilon_\kappa^{-2}\to 0$. An integration
by parts in the first term in \eqref{9} gives
\begin{equation}\label{3}
\int_{B_1(0)}-2\eta w_\kappa
\nabla w_{\kappa}\nabla\eta-\frac{3}{2}|\nabla
w_\kappa|^2\eta^2-|\nabla
w_\kappa|^2\eta\frac{\partial\eta}{\partial
x_n}x_n+\sum_{i=1}^n \frac{\partial w_\kappa}{\partial
x_n}\frac{\partial w_\kappa}{\partial
x_i}\frac{\partial}{\partial x_i}\left(\eta^2x_n\right)\to0.
\end{equation}

Next we analyze the convergence in \eqref{3} term by term. By the
weak convergence of $\nabla w_\kappa$ and strong convergence of
$w_\kappa$ in $L^2(B_1(0))$,
\[\int_{B_1(0)}\eta w_\kappa \nabla w_{\kappa}\nabla\eta\to\int_{B_1(0)}\eta w\nabla w\nabla\eta.\]
By the weak convergence of $|\nabla w_\kappa|^2dx$,
\[\int_{B_1(0)}|\nabla w_\kappa|^2\eta^2\to\int_{B_1(0)}|\nabla w|^2\eta^2+\int_{B_1(0)}\eta^2d\mu,\]
and
\begin{eqnarray*}
\int_{B_1(0)}|\nabla w_\kappa|^2\eta\frac{\partial\eta}{\partial x_n}x_n&\to&\int_{B_1(0)}|\nabla w|^2\eta\frac{\partial\eta}{\partial x_n}x_n+\int_{B_1(0)}\eta\frac{\partial\eta}{\partial x_n}x_nd\mu\\
&=&\int_{B_1(0)}|\nabla w|^2\eta\frac{\partial\eta}{\partial x_n}x_n.
\end{eqnarray*}
Here we have used Lemma \ref{lem support of defect measure} and the fact that $\eta\frac{\partial\eta}{\partial x_n}x_n=0$ on $\{x_n=0\}$. Similarly,
\begin{equation*}
\int_{B_1(0)}\sum_{i=1}^n \frac{\partial w_\kappa}{\partial x_n}\frac{\partial w_\kappa}{\partial x_i}\frac{\partial}{\partial x_i}\left(\eta^2x_n\right)\to\int_{B_1(0)}\sum_{i=1}^n \frac{\partial w}{\partial x_n}\frac{\partial w}{\partial x_i}\frac{\partial}{\partial x_i}\left(\eta^2x_n\right)+\int_{B_1(0)}\eta^2d\mu^n,
\end{equation*}
where $\mu^n$ is the weak limit of the measures $(\frac{\partial w_\kappa}{\partial x_n})^2dx-(\frac{\partial w}{\partial x_n})^2dx$.

Substituting these into \eqref{3} we get
\begin{eqnarray}\label{4}
&&\int_{B_1(0)}-2\eta w \nabla w\nabla\eta-\frac{3}{2}|\nabla w|^2\eta^2
-|\nabla w|^2\eta\frac{\partial\eta}{\partial x_n}x_n+\sum_{i=1}^n \frac{\partial w}{\partial x_n}\frac{\partial w}{\partial x_i}\frac{\partial}{\partial x_i}\left(\eta^2x_n\right)\nonumber\\
&&\ \ \ \ \ \ \ \ \ \ \ \ \ \ \ \ \ \ \ -\frac{3}{2}\int_{B_1(0)}\eta^2d\mu+\int_{B_1(0)}\eta^2d\mu^n=0.
\end{eqnarray}

Since $w$ is a harmonic function, an integration by parts gives
\begin{equation}\label{5}
\int_{B_1(0)}-2\eta w \nabla w\nabla\eta=\int_{B_1(0)}|\nabla w|^2\eta^2.
\end{equation}
Note that $w$ is smooth. Then by standard domain variation
arguments, we also have a stationary condition for $w$, which
says, for any smooth vector field $Y$ with compact support in
$B_1(0)$,
\[\int_{B_1(0)}|\nabla w|^2\mbox{div}Y-2DY(\nabla w,\nabla w)=0.\]
By taking  $Y=(0,\cdots,0, \eta^2x_n)$ in this equality, we obtain
\begin{equation*}
\int_{B_1(0)}|\nabla w|^2\left(\eta^2+2\eta\frac{\partial\eta}{\partial x_n}x_n\right)-2\sum_{i=1}^n\frac{\partial w}{\partial x_n}\frac{\partial w}{\partial x_i}\frac{\partial}{\partial x_i}\left(\eta^2x_n\right)=0.
\end{equation*}
Substituting this and \eqref{5} into \eqref{4}, we get
\begin{equation}\label{6}
-\frac{3}{2}\int_{B_1(0)}\eta^2d\mu+\int_{B_1(0)}\eta^2d\mu^n=0.
\end{equation}
By the weak convergence of $w_\kappa$ in $H^1(B_1(0))$, we also have
\[|\nabla(w_\kappa-w)|^2dx\rightharpoonup\mu,\]
\[\left(\frac{\partial w_\kappa}{\partial x_n}-\frac{\partial w}{\partial x_n}\right)^2dx\rightharpoonup\mu^n.\]
From these we see for any $\eta\in C_0^\infty(B_1(0))$,
\[\int_{B_1(0)}\eta^2d\mu\geq\int_{B_1(0)}\eta^2d\mu^n.\]
 Substituting this into \eqref{6}, we see
\[\int_{B_1(0)}\eta^2d\mu\leq 0.\]
Since $\mu$ is a positive Radon measure, this implies that $\mu=0$, and hence the strong convergence of $w_\kappa$ in $H^1_{loc}(B_1(0))$.

With these preliminary analysis we come to the
\begin{proof}[Proof of Theorem \ref{thm decay estimate}]
Note that $w$ is a harmonic function satisfying
\[\int_{B_1(0)}|\nabla w|^2\leq 1.\]
By standard interior gradient estimates for harmonic functions, there exists a constant $C_1(n)$ depending only on the dimension $n$, such that for any $r\in(0,1/2)$,
\[\int_{B_r(0)}|\nabla w-\nabla w(0)|^2\leq C_1(n)r^{n+2}.\]
Here we have used the fact that each component of $\nabla w$ is
harmonic. By the mean value property for harmonic functions, there
exists another constant $C_2(n)$ still depending only on the
dimension $n$, such that
\[|\nabla w(0)|\leq C_2(n).\]
With these choices, now we fix the constant $C(n)$ in
\eqref{assumption 4} to be $2C_2(n)$.

Fix a $\theta\in(0,1/2)$ so that
$$2C_1(n)\theta^2\leq\frac{1}{4}.$$
Then by the strong convergence of $w_\kappa$ to $w$ in $H^1(B_{1/2})$, for $\kappa$ large,
\[\theta^{-n}\int_{B_\theta(0)}|\nabla w_\kappa-\nabla w(0)|^2
\leq\theta^{-n}\int_{B_\theta(0)}|\nabla w-\nabla w(0)|^2+\frac{1}{4}\leq
\frac{1}{2}.\] In other words,
\[\theta^{-n}\int_{B_\theta(0)}|\nabla(u_\kappa-v_\kappa)-\left(e+\varepsilon_\kappa\nabla w(0)\right)|^2\leq \frac{1}{2}\varepsilon_\kappa^2.\]
By our construction, $e+\varepsilon_\kappa\nabla w(0)$
satisfies \eqref{assumption 4}. The above inequality contradicts
\eqref{assumption 2} and we finish the proof of Theorem \ref{thm
decay estimate}.
\end{proof}

\bibliographystyle{amsplain}

\end{document}